\documentclass{amsart}
\usepackage{amssymb,latexsym,amsmath,amsfonts,hhline}
\usepackage{pdfsync}
\usepackage{ucs}
\usepackage{amssymb}
\usepackage{amsthm}
\usepackage{amsmath}
\usepackage{latexsym}
\usepackage[cp1251]{inputenc}
\usepackage[english]{babel}
\usepackage[mathcal]{eucal}
\usepackage{graphicx}
\usepackage{wrapfig}
\usepackage{caption}
\usepackage{subcaption}

\usepackage[left=3cm,right=3cm,top=3cm,bottom=3cm,bindingoffset=0cm]{geometry}

\makeatletter
\renewcommand{\subsection}{\@startsection{subsection}{2}{0mm}{-2mm}{-2mm}{\bf\normalsize}}

\makeatother

\newtheorem{formula}{}[section]
\newtheorem{definition}[formula]{Definition}
\newtheorem{corollary}[formula]{Corollary}
\newtheorem{remark}[formula]{Remark}
\newtheorem{lemma}[formula]{Lemma}
\newtheorem{theorem}[formula]{Theorem}
\newtheorem{prop}{Proposition}[section]

\theoremstyle{definition}

\theoremstyle{remark}

\def\thrm{\begin{theorem}}
\def\thrml#1{\begin{theorem}\label{#1}}
\def\ethrm{\end{theorem}}
\def\rmrk{\begin{remark}}
\def\rmrkl#1{\begin{remark}\label{#1}}
\def\ermrk{\end{remark}}
\def\dfntn{\begin{definition}}
\def\dfntnl#1{\begin{definition}\label{#1}}
\def\edfntn{\end{definition}}
\def\nmrt{\begin{enumerate}}
\def\enmrt{\end{enumerate}}

\def\qtn{\begin{equation}}
\def\qtnl#1{\begin{equation}\label{#1}}
\def\eqtn{\end{equation}}
\def\lmm{\begin{lemma}}
\def\lmml#1{\begin{lemma}\label{#1}}
\def\elmm{\end{lemma}}
\def\crllr{\begin{corollary}}
\def\crllrl#1{\begin{corollary}\label{#1}}
\def\ecrllr{\end{corollary}}
\def\css{\begin{cases}}
\def\ecss{\end{cases}}

\begin{document}
\title{On Schur 3-groups}
\author{Grigory Ryabov}
\address{Novosibirsk State University, 2 Pirogova St., 630090 Novosibirsk, Russia}
\email{gric2ryabov@gmail.com}
\thanks{The work is supported by the Russian Foundation for Basic Research (project 13-01-00505).}
\date{}

\begin{abstract}
Let $G$ be a finite group. If $\Gamma$ is a permutation group with  $G_{right}\leq\Gamma\leq Sym(G)$ and $\mathcal{S}$ is the set of orbits of the stabilizer of the identity $e=e_{G}$ in $\Gamma$, then the $\mathbb{Z}$-submodule $\mathcal{A}(\Gamma,G)=Span_{\mathbb{Z}}\{\underline{X}:\ X\in\mathcal{S}\}$ of the group ring $\mathbb{Z} G$ is an $S$-ring as it was observed by Schur. Following P\"{o}schel an $S$-ring $\mathcal{A}$ over $G$ is said to be \emph{schurian} if there exists a suitable permutation group $\Gamma$ such that $\mathcal{A}=\mathcal{A}(\Gamma,G)$. A finite group $G$ is called a \emph{Schur group} if every $S$-ring over $G$ is schurian.  We prove that the groups $M_{3^n}=\langle a,b\;|\:a^{3^{n-1}}=b^3=e,a^b=a^{3^{n-2}+1}\rangle$, where $n\geq3$, are not Schur. Modulo previously obtained results, it follows that every Schur $p$-group is abelian whenever $p$ is an odd prime.
\end{abstract}

\maketitle

\section{Introduction}
	
	Let $G$ be a finite group, $e$  the identity element of $G$. Let  $\mathbb{Z}G$ be the integer group ring. Given $X\subseteq G$,  denote the element $\sum_{x\in X} {x}$ by $\underline{X}$ .
	\begin{definition}
A subring  $\mathcal{A}$ of  $\mathbb{Z} G$ is called an \emph{$S$-ring} over $G$ if there exists a partition $\mathcal{S}=\mathcal{S}(\mathcal{A})$ of $G$ such that:

 $(1)$ $\left\{e\right\}\in\mathcal{S}$,

 $(2)$ $X\in\mathcal{S}\ \Rightarrow\ X^{-1}\in\mathcal{S}$,

 $(3)$ $\mathcal{A}=Span_{\mathbb{Z}}\{\underline{X}:\ X\in\mathcal{S}\}$.
\end{definition}
The elements of this partition are called \emph{the basic sets}  of the $S$-ring $\mathcal{A}$.

Let $\Gamma$ be a subgroup of $Sym(G)$ that contains  the subgroup of right shifts $G_{right}=\{x\mapsto xg,~x\in G:g\in G\}$. Let $\Gamma_e$ stand for the stabilizer of $e$ in $\Gamma$ and  $Orb(\Gamma_e,G)$ stand for the set of all orbits $\Gamma_e$ on $G$. As I. Schur proved in  \cite{Schur}, the $\mathbb{Z}$-submodule
$$\mathcal{A}=\mathcal{A}(\Gamma,~G)=Span_{\mathbb{Z}}\left\{\underline{X}:~X\in Orb(\Gamma_e,~G)\right\},$$
is an $S$-ring over $G$.

\begin{definition}
An $S$-ring $\mathcal{A}$ over  $G$ is called \emph{schurian} if $\mathcal{A}=\mathcal{A}\left(\Gamma,~G\right)$ for some $\Gamma$ with $G_{right}\leq \Gamma \leq Sym(G)$.
	\end{definition}
\begin{definition}
A finite group $G$ is called a \emph{Schur} group if every $S$-ring over $G$ is schurian.
\end{definition}
The problem of determining all Schur groups was suggested by R. P\"{o}schel in \cite{Po} about 40 years ago. He proved that a $p$-group, $p>3$, is  Schur if and only if it is cyclic. Using this result R. P\"{o}schel and M. Klin solved the isomorphism problem for circulant graphs with $p^n$ vertices, where $p$ is an odd prime and $n\geq 1$ is an integer \cite{KP}. Only 30 years later all cyclic Schur groups were classified  in  \cite{EKP1}. Strong necessary conditions of schurity for abelian groups were recently proved in \cite{EKP2}.
	
	All Schur groups of order $\leq 62$ were found by  computer calculations \cite{Fi,Ziv}. It turned out that there are non-abelian Schur groups. However, except for P\"{o}schel's result about $p$-groups, there were no general results on non-abelian Schur groups.
	Recently it was proved  \cite{PV} that every Schur group $G$ is solvable of derived length at most 2  and the number of distinct prime divisors of the order of $G$ does not exceed~$7$. In the same article  it was proved that only the groups $M_{3^n}=\langle a,b\;|\:a^{3^{n-1}}=b^3=e,a^b=a^{3^{n-2}+1}\rangle,~n\geq 3,$ might be non-abelian Schur $3$-groups. 

	The main result of this paper is  the following

\begin{theorem} \label{main}
The groups $M_{3^n}=\langle a,b\;|\:a^{3^{n-1}}=b^3=e,a^b=a^{3^{n-2}+1} \rangle$, $n\geq 3,$ are not Schur.
\end{theorem}
	
It is worth noting that the case of Schur $2$-groups was very recently analyzed in \cite{MP2} by M. Muzychuk and I. Ponomarenko. The author is grateful to both of them for fruitful discussions on the subject matters.

	From the above discussion and Theorem \ref{main}, we immediately obtain the following statement.
	
	\begin {corollary}
   Every Schur $p$-group is abelian whenever $p$ is an odd prime.
	\end{corollary}

	\section{Preliminaries}
	
	In this section we recall some definitions and facts about $S$-rings and Cayley schemes. All of them are taken from \cite{MP}, so we skip further references here.
	
	\begin{definition}
Let $G$ be a finite group, $\mathcal{R}$  a family of binary relations on $G$.
The pair \emph{$\mathcal{C}=\left(G,\mathcal{R}\right)$} is called a \emph{Cayley scheme} over $G$ if the following properties are satisfied:

$(1)$  $\mathcal{R}$ forms a partition of the set $G\times G$;

$(2)$  $Diag\left(G\times G\right)\in\mathcal{R}$;

$(3)$  $\mathcal{R}=\mathcal{R}^*$, i.\,e., if $R\in\mathcal{R}$ then $R^*=\{(h,g)\mid (g,h)\in R\}\in\mathcal{R}$;

$(4)$  if $R,~S,~T\in\mathcal{R}$ and  $(f,g)\in T$, then the number $|\{h\in G:(f,h)\in R,~(h,g)\in S\}|$ does not depend on the choice of $(f,g)$.

\end{definition}
	
	Let $\mathcal{A}$ be an  $S$-ring over $G$. We associate each basic set  $X\in \mathcal{S}(\mathcal{A})$  with the binary relation  $\{(a,xa) \mid a\in G, x\in X\}\subseteq G\times G$ and denote it by $R(X)$. The set of all such binary relations forms a partition $\mathcal{R}(\mathcal{S}(\mathcal{A}))$ of  $G\times G$.
	\begin{lemma}
$\mathcal{C}(\mathcal{A})=(G,\mathcal{R}(\mathcal{S}(\mathcal{A})))$ is a Cayley scheme over $G$. The map $\mathcal{A}\mapsto \mathcal{C}(\mathcal{A})$ is a bijection between $S$-rings  and Cayley schemes over $G$.
\end{lemma}
	
	The relation $R(X)$ is called \emph{the basic relation} of the scheme $\mathcal{C}(\mathcal{A})$ corresponding to $X$.
	
	\begin{definition}
Cayley schemes  $\mathcal{C}=(G,\mathcal{R})$ and $\mathcal{C}^{'}=(G^{'},\mathcal{R}^{'})$ are called \emph{isomorphic} if there exists a bijection $f:G\rightarrow G^{'}$ such that  $\mathcal{R}^{'}=\mathcal{R}^f$, where $\mathcal{R}^f=\{R^f:~R\in\mathcal{R}\}$ and $R^f=\{(a^f,~b^f):~(a,~b)\in R\}$.
	\end{definition}
	
	The set of all isomorphisms from $\mathcal{C}$ onto $\mathcal{C}^{'}$  is denoted by $\operatorname{Iso}(\mathcal{C},\mathcal{C}^{'})$. The group $\operatorname {Iso}(\mathcal{C})=\operatorname{Iso}(\mathcal{C},\mathcal{C})$  of all  isomorphisms of  $\mathcal{C}$ onto itself contains the normal
subgroup
 $$\operatorname {Aut}(\mathcal{C})=\{f\in \operatorname {Iso}(\mathcal{C}):~R^f=R,~R\in\mathcal{R}\}$$
called \emph{the automorphism group} of  $\mathcal{C}$.
	
\begin{definition}
	Two $S$-rings $\mathcal{A}$  over $G$ and $\mathcal{A}^{'}$ over $G^{'}$ are called \emph{isomorphic} if there exists an isomorphism of the corresponding Cayley schemes $\mathcal{C}(\mathcal{A})$ and $\mathcal{C^{'}}(\mathcal{A^{'}})$ taking  the identity element of $G$  to the identity element of $G^{'}$. It is called \emph{the isomorphism} from $\mathcal{A}$ to $\mathcal{A}^{'}$.
\end{definition}
	
	We denote the set of all isomorphisms from  $\mathcal{A}$ onto $\mathcal{A}^{'}$ by  $\operatorname{Iso}(\mathcal{A},\mathcal{A}^{'})$.  The group $\operatorname{Iso}(\mathcal{A})=\operatorname {Iso}(\mathcal{A},\mathcal{A})$ contains the normal subgroup
	$$\operatorname {Aut}(\mathcal{A})=\{f\in \operatorname {Iso}(\mathcal{A}):~R(X)^f=R(X),~X\in \mathcal{S}(\mathcal{A})\}$$
called \emph{the automorphism group} of  $\mathcal{A}$.
	\begin{lemma}
	Suppose that $\mathcal{A}$ is an $S$-ring, $\mathcal{C}\left(\mathcal{A}\right)$ is the corresponding Cayley scheme. Then
	$$\operatorname {Iso}(\mathcal{A})=\operatorname{Iso}(\mathcal{C(\mathcal{A})})_{e},~\operatorname {Iso}(\mathcal{C}(\mathcal{A}))=G_{right}\operatorname{Iso}(\mathcal{A}),$$
	$$\operatorname {Aut}(\mathcal{A})=\operatorname {Aut}(\mathcal{C}(\mathcal{A}))_{e},~\operatorname {Aut}(\mathcal{C}(\mathcal{A}))=G_{right}\operatorname {Aut}(\mathcal{A}).$$
 \end{lemma}

\begin{lemma}\label{schuriancond}
An $S$-ring $\mathcal{A}$ over $G$ is schurian if and only if   $\mathcal{S}(\mathcal{A})=Orb(\operatorname {Aut}(\mathcal{A}),G)$.
\end{lemma}

\begin{definition}
Let $\mathcal{A}$ be an $S$-ring over $G$. A subgroup $H \leq G$ is called \emph{an $\mathcal{A}$-subgroup} if $\underline{H}\in \mathcal{A}$.
\end{definition}

\begin{definition}
Let $L,U$ be subgroups of a group $G$ and $L$ be normal in $U$. A section $U/L$ of $G$ is called an \emph{$\mathcal{A}$-section} if $U$ and $L$ are $\mathcal{A}$-subgroups.
\end{definition}

\begin{lemma}\label{ssection}
Let $D=U/L$ be an $\mathcal{A}$-section. Then the module
$$\mathcal{A}_D=Span_{\mathbb{Z}}\left\{\underline{X}^{\pi}:~X\in\mathcal{S}\left(\mathcal{A}\right),~X\subseteq U\right\},$$
where $\pi:U\rightarrow U/L$ is the canonical homomorphism, is an $S$-ring over $D$.

In addition, if $\mathcal{A}$ is schurian, then $\mathcal{A}_D$ is schurian too.
\end{lemma}

\begin{lemma}\label{block}
Let  $\mathcal{A}$ be a schurian $S$-ring, $\mathcal{A}=\mathcal{A}\left(\Gamma,~G\right)$. Then $\Gamma$  is imprimitive if and only if there exists a non-trivial proper $\mathcal{A}$-subgroup $H$ of $G$. In this case  the left cosets of $H$ form the non-trivial block system  of $\Gamma$.
\end{lemma}

\section{Non-schurity of $M_{27}$}

Let $G=M_{27}=\langle a,b\;|\:a^{3^2}=b^3=e,a^b=a^{4} \rangle$. Denote the subgroup $\langle a^3b \rangle=\{e,~a^3b,~a^6b^2\}$ by  $U$. Consider the sets
$$Z_0=\left\{e\right\},$$
$$Z_1=\left\{a^3b,~a^6b^2\right\}=U\setminus \left\{e\right\},$$
$$Z_2=\left\{a,~a^3,~a^6,~a^8,~a^4b^2,~a^7b,~a^8b,~a^8b^2\right\},$$
$$Z_3= G\setminus \left(Z_0 \cup Z_1 \cup  Z_2\right).$$

The sets $Z_0,~Z_1,~Z_2,$ and $Z_3$  form the partition of $G$, which is denoted by $\mathcal{S}$. Note that  $Z_i=Z_i^{-1},~i=0,\ldots,3$.

\begin{lemma}
The $\mathbb{Z}$-module $\mathcal{A}$ spanned by the elements $\xi_i=\underline{Z_i},~i=0,\ldots,3,$ is a commutative  $S$-ring  over~$G$.
	\end{lemma}

\begin{proof}
The commutativity of $\mathcal{A}$ immediately follows from the fact that each class of the partition is closed with respect to taking inverse. The computations in the group ring of $G$ show that

$$\xi_0\xi_i= \xi_i\xi_0=\xi_i;$$

  $$\xi_1\xi_1=2\xi_0+\xi_1,$$
	$$\xi_1\xi_2=\xi_2\xi_1=\xi_3,$$
	$$\xi_1\xi_3=\xi_3\xi_1=\xi_3+2\xi_2;$$
	
	$$\xi_2\xi_2=8\xi_0+\xi_2+3\xi_3,$$
	$$\xi_2\xi_3=\xi_3\xi_2=8\xi_1+6\xi_2+4\xi_3;$$
	
	$$\xi_3\xi_3=16\xi_0+8\xi_1+8\xi_2+10\xi_3.$$
	
\end{proof}

\begin{prop}\label{M27}
The $S$-ring  $\mathcal{A}$ is not schurian.
\end{prop}

\begin{proof}
Suppose on the contrary that  $\mathcal{A}$ is schurian. Then it follows from Lemma~\ref{schuriancond} that $\mathcal{A}=\mathcal{A}(KG_{right},G)$, where $K =\operatorname{Aut}(\mathcal{A})$ and the sets $Z_i,~i=0,\ldots,3,$ are the orbits of $K$.

\begin{lemma}\label{Ublock}
 The subgroup $U$ is an $\mathcal{A}$-subgroup. The left cosets of $U$ are the blocks of $K$.
	\end{lemma}
\begin{proof}
Note that $U=Z_0 \cup Z_1$. Then $\underline{U}\in \mathcal{A}$. Therefore $U$ is an $\mathcal{A}$-subgroup and we are done by Lemma \ref{block}.
\end{proof}

The elements $b$ and $a^2$ belong to $Z_3$. Since the latter is a $K$-orbit, there exists  $\alpha\in K$  taking  $b$ to $a^2$.

Let $\mathcal{C}(\mathcal{A})$ be the Cayley scheme corresponding to $\mathcal{A}$. Then $\alpha$ is an automorphism of $\mathcal{C}(\mathcal{A})$. Therefore $\alpha$ takes the neighborhood $\{a^3b^2,a^6\}$ of $b$ in $R(Z_1)$ to the neighborhood $\{a^8b,a^5b^2\}$ of $a^2$ in $R(Z_1)$. However, the elements $a^6$ and $a^8b$ lie in  $Z_2$, the elements $a^3b^2$ and $a^5b^2$ lie in $Z_3$. Thus, since $\alpha$ preserves $Z_2$ and $Z_3$, we have
$${(a^6)}^{\alpha}=a^8b;\eqno(1)$$
$${(a^3b^2)}^{\alpha}=a^5b^2.\eqno(2)$$
It follows from (1) and Lemma \ref{Ublock}  that ${(a^3b^2U)}^{\alpha}=a^5U$. This contradicts (2) because ${(a^3b^2)}^{\alpha}=a^5b^2 \notin a^5U$.
\end{proof}

\section{Non-schurity of $M_{3^n},~n\geq 4$}

Let $G=M_{3^n}=\langle a,b\;|\:a^{3^{n-1}}=b^3=e,a^b=a^{3^{n-2}+1} \rangle$ and $n\geq 4$. Put $A=\langle a\rangle$, $B=\langle b\rangle$, $c=a^{3^{n-2}}$, $C=\langle c\rangle$, and $H=C\times B$. Obviously, $ |A|=3^{n-1},~|B|=3,~|C|=3,$ and $|H|=9$.

Consider the sets
$$Z_0=\{e\},$$
$$Z_1=\{b, b^2\},$$
$$Z_2=\{c,c^2\},$$
$$Z_3=\{cb, cb^2, c^2b, c^2b^2\} = H\setminus (Z_0 \cup Z_1 \cup  Z_2).$$
$$Z_4= a\{e, cb, c^2b^2\}\cup a^{-1}\{e, c^2b, cb^2\},$$
$$Z_5=(aH\cup a^{-1}H)\setminus Z_4,$$
$$X_k = a^{3k}C\cup a^{-3k}C,~Y_k=(a^{3k}H\cup a^{-3k}H)\setminus X_k,~k = 1,\ldots,\frac{3^{n-3}-1}{2}, $$
$$T_j=a^jH\cup a^{-j}H,~ j = 2,~4,~5,\ldots,\frac{3^{n-2}-1}{2},~ j\not\equiv 0\mod 3.$$

Note that the sets $Z_i,~X_k,~Y_k,~T_j$  form the partition of $G$, denote it by $\mathcal{S}$. It is easy to check that $Z_i=Z_i^{-1},~X_k=X_k^{-1},~Y_k=Y_k^{-1},~T_j=T_j^{-1}$.

\begin{lemma}
 The $\mathbb{Z}$-module $\mathcal{A}$ spanned by the elements $\xi_i=\underline{Z_i},~\theta_k=\underline{X_k},~\psi_k=\underline{Y_k},~\varphi_j=\underline{T_j}$ is a commutative  $S$-ring  over group $G$.
	\end{lemma}

\begin{proof}
The commutativity of $\mathcal{A}$ immediately follows from the fact that each class of the partition is closed with respect to taking inverse. The  computations in the group ring of $G$ show that

$$\xi_0\xi_i= \xi_i,~\xi_0\theta_k=\theta_k,~\xi_0\psi_k=\psi_k,~\xi_0\varphi_j=\varphi_j;$$

  $$\xi_1\xi_1=2\xi_0+\xi_1,$$
	$$\xi_1\xi_2=\xi_2\xi_1=\xi_3,$$
	$$\xi_1\xi_3=\xi_3\xi_1=\xi_3+2\xi_2,$$
	$$\xi_1\xi_4=\xi_4\xi_1=\xi_5,$$
	$$\xi_1\xi_5=\xi_5\xi_1=\xi_5+2\xi_4,$$
	$$\xi_1\theta_k=\theta_k\xi_1=\psi_k,$$
	$$\xi_1\psi_k=\psi_k\xi_1=\psi_k+2\theta_k,$$
	$$\xi_1\varphi_j=\varphi_j\xi_1=2\varphi_j;$$
	
	$$\xi_2\xi_2=2\xi_0+\xi_2,$$
	$$\xi_2\xi_3=\xi_3\xi_2=\xi_3+2\xi_1,$$
	$$\xi_2\xi_4=\xi_4\xi_2=\xi_5,$$
	$$\xi_2\xi_5=\xi_5\xi_2=2\xi_4+\xi_5,$$
	$$\xi_2\theta_k=\theta_k\xi_2=2\theta_k,$$
	$$\xi_2\psi_k=\psi_k\xi_2=2\psi_k,$$
	$$\xi_2\varphi_j=\varphi_j\xi_2=2\varphi_j;$$
	
	$$\xi_3\xi_3=\xi_3+2\xi_1+2\xi_2+4\xi_0,$$
	$$\xi_3\xi_4=\xi_4\xi_3=\xi_5+2\xi_4,$$
	$$\xi_3\xi_5=\xi_5\xi_3=3\xi_5+2\xi_4,$$
	$$\xi_3\theta_k=\theta_k\xi_3=2\psi_k,$$
	$$\xi_3\psi_k=\psi_k\xi_3=2\psi_k+4\theta_k,$$
	$$\xi_3\varphi_j=\varphi_j\xi_3=4\varphi_j;$$
	
	$$\xi_4\xi_4=\varphi_2+6\xi_0 +3\xi_3,$$
	$$\xi_4\xi_5=\xi_5\xi_4=2\varphi_2+6\xi_1+6\xi_2+3\xi_3,$$
	$$\xi_4\theta_k=\theta_k\xi_4=\varphi_{3k+1}+\varphi_{3k-1},$$
	$$\xi_4\psi_k=\psi_k\xi_4=2\varphi_{3k+1}+2\varphi_{3k-1},$$
	$$\xi_4\varphi_j=\varphi_j\xi_4=3\varphi_{j+1}+3\theta_l+3\psi_l,~ j-1=3l,$$
	$$\xi_4\varphi_j=\varphi_j\xi_4=3\varphi_{j-1}+3\theta_l+3\psi_l,~ j+1=3l;$$
	
	$$\xi_5\xi_5=4\varphi_2+6\xi_1+6\xi_2+9\xi_3+12\xi_0,$$
	$$\xi_5\theta_k=\theta_k\xi_5=2\varphi_{3k+1}+2\varphi_{3k-1},$$
	$$\xi_5\psi_k=\psi_k\xi_5=4\varphi_{3k+1}+4\varphi_{3k-1},$$
	$$\xi_5\varphi_j=\varphi_j\xi_5=6\varphi_{j+1}+6\theta_l+6\psi_l,~ j-1=3l,$$
	$$\xi_5\varphi_j=\varphi_j\xi_5=6\varphi_{j-1}+6\theta_l+6\psi_l,~ j+1=3l;$$
	
	$$\theta_k\theta_l=\theta_l\theta_k=\theta_{k+l}+\theta_{k-l},~k\neq l,$$
	$$\theta_k\theta_k=3\theta_{2k}+6\xi_0+6\xi_2,$$
	$$\theta_k\psi_l=\psi_l\theta_k=\psi_{k+l}+\psi_{k-l},~k\neq l,$$
	$$\theta_k\psi_k=\psi_k\theta_k=3\psi_{2k}+6\xi_1+6\xi_3,$$
	$$\theta_k\varphi_j=\varphi_j\theta_k=3\varphi_{3k+j}+3\varphi_{3k-j};$$
	
	$$\psi_k\psi_l=\psi_l\psi_k=2\theta_{k+l}+2\theta_{k-l}+\psi_{k+l}+\psi_{k-l},~k\neq l,$$
	$$\psi_k\psi_k=3\psi_{2k}+6\theta_{2k}+6\xi_1+12\xi_2+6\xi_3+12\xi_0,$$
	$$\psi_k\varphi_j=\varphi_j\psi_k=6\varphi_{3k+j}+6\varphi_{3k-j};$$
	
	$$\varphi_i\varphi_j=\varphi_j\varphi_i=9\varphi_{i+j}+9\theta_k+9\psi_k,~i-j=3k,$$
	$$\varphi_i\varphi_j=\varphi_j\varphi_i=9\varphi_{i-j}+9\theta_k+9\psi_k,~i+j=3k,$$
	$$\varphi_j\varphi_j=9\varphi_{2j}+18\xi_0+18\xi_1+18\xi_2+18\xi_3.$$
	
Let us, for example, check that $\varphi_j\varphi_j=9\varphi_{2j}+18\xi_0+18\xi_1+18\xi_2+18\xi_3$. Since $G^{'}\leq H$, the subgroup $H$ is normal in $G$ and $gH=Hg$ for every $g\in G$. So $\varphi_j\varphi_j=(a^j\underline{H}+a^{-j}\underline{H})^2=a^{2j}(\underline{H})^2+a^{-2j}(\underline{H})^2+2(\underline{H})^2=9a^{2j}\underline{H}+9a^{-2j}\underline{H}+18\underline{H}=9\varphi_{2j}+18\xi_0+18\xi_1+18\xi_2+18\xi_3.$

\end{proof}

\begin{prop}\label{M3n}
The $S$-ring  $\mathcal{A}$ is not schurian.
\end{prop}

\begin{proof}
Assume the contrary. Then it follows from Lemma~\ref{schuriancond}  that $~\mathcal{A}=\mathcal{A}(KG_{right},G)$ where $K =\operatorname{Aut}(\mathcal{A})$ and the sets  $Z_i,~X_k,~Y_k,~T_j$ are the orbits of $K$.

\begin{lemma}\label{KBlocks}
 The groups $C$, $B$, $H$ are $\mathcal{A}$-subgroups. The left cosets of $C$, $B$, $H$ are the blocks of $K$.
	\end{lemma}
\begin{proof}
Note that $B=Z_0 \cup Z_1,~ C=Z_0 \cup Z_2,~ H=Z_0 \cup Z_1\cup Z_2\cup Z_3$. Then $\underline{C},~\underline{B},~\underline{H}\in \mathcal{A}$. Therefore $C$, $B$, $H$ are $\mathcal{A}$-subgroups and it follows from Lemma~\ref{block}  that the left cosets of $C$, $B$, $H$ are the blocks of~$K$.
\end{proof}

Since $H$ is  a normal $\mathcal{A}$-subgroup, following Lemma~\ref{ssection} we can form an  $S$-ring over $G/H$
$$\mathcal{A}_{G/H}=Span_{\mathbb{Z}}\left\{\underline{X}^{\pi}:~X\in\mathcal{S}\left(\mathcal{A}\right)\right\},$$
where $\pi:G\rightarrow G/H$ is a canonical homomorphism.

\begin{lemma}\label{2factor}
 The automorphism group of the $S$-ring $\mathcal{A}_{G/H}$ is of order $2$. Its non-identity element permutes $a^iH$ and $a^{-i}H$.
	\end{lemma}

\begin{proof}
The basic sets of $\mathcal{A}_{G/H}$ are of the form
$$\left\{H\right\}, \left\{a^iH, a^{-i}H\right\}, i=1,\ldots,\frac{3^{n-2}-1}{2}.$$
Let $\mathcal{C}(\mathcal{A}_{G/H})$ be the Cayley scheme corresponding to $\mathcal{A}_{G/H}$. The basic relation of $\mathcal{C}(\mathcal{A}_{G/H})$ corresponding to the  basic set $\left\{aH, a^{-1}H\right\}$ is a cycle of length $3^{n-2}$. The conclusion of the lemma is a direct consequence of  the fact that the automorphism group of the graph of undirected cycle is dihedral, and the one point stabilizer of it is of order $2$.
\end{proof}

Consider the action of the elements from $K$ on the set $Z_3$. The automorphisms of $\mathcal{A}$ do not contain in their cyclic structure cycles of length $3$ and  $4$ consisting of the elements from $Z_3$ because by Lemma~\ref{KBlocks} the left cosets $cB,~c^2B$ and $bC,~b^2C$ are the blocks of $K$. Since $Z_3$ is an orbit of $K$, it follows that $K$ acts on $Z_3$ as the Klein four-group $K_4$. Since $Z_1$ and $Z_2$ are $K$-orbits of length $2$, each permutation of $K$ can be written as follows:

$$ \gamma,\eqno(1) $$
$$ (b,~b^2) (cb,~cb^2)(c^2b,~c^2b^2)\gamma, \eqno(2)$$
$$ (c,~c^2) (cb,~c^2b)(cb^2,~c^2b^2)\gamma,\eqno(3)$$
$$ (c,~c^2)(b,~b^2) (cb,~c^2b^2)(cb^2,~c^2b)\gamma ,\eqno(4)$$
where  $\gamma\in K$ acting on $H$ trivially.

Let $\mathcal{C}(\mathcal{A})$ be the Cayley scheme corresponding to $\mathcal{A}$. Below we list  the elements adjacent to  $e,~b,~b^2,~c,~c^2$ in $R(Z_4)$:

$$e:a,~ cab,~c^2ab^2,~a^{-1},~ ca^{-1}b^2,~ c^2a^{-1}b;$$
$$b:ab,~ cab^2,~c^2a,~a^{-1}b,~ ca^{-1},~ c^2a^{-1}b^2;$$
$$b^2:ab^2,~ ca,~c^2ab,~a^{-1}b^2,~ ca^{-1}b,~ c^2a^{-1};$$
$$c:ab^2,~ ca,~c^2ab,~a^{-1}b,~ ca^{-1},~ c^2a^{-1}b^2;$$
$$c^2:ab,~ cab^2,~c^2a,~a^{-1}b^2,~ ca^{-1}b,~ c^2a^{-1}.$$

Denote the set of all elements adjacent to $q\in G$ in $R(Z_4)$  by $L_q$. Let $\alpha\in K$ is of type (4). Then $b^{\alpha}=b^2$ and $(c^2)^{\alpha}=c$. The element $ab$ is adjacent to $b$ and $c^2$ in $R(Z_4)$. Therefore $(ab)^{\alpha}$ is adjacent to $b^2$ and $c$ in $R(Z_4)$. We conclude that $(ab)^{\alpha}\in L_{b^2}\cap L_c=\{ab^2,~ca,~c^2ab\}$. Similarly an element of type (2) takes $ab$ to one of the elements
$ca^{-1}b,~a^{-1}b^2,~c^2a^{-1}$ and an element of type (3) takes $ab$ to one of the elements $a^{-1}b,~ca^{-1},~c^2a^{-1}b^2.$ This means that $(ab)^{\alpha} \in aH\cap (aH)^{\alpha}$ for $\alpha\in K$ of types (1) and (4), $(ab)^{\alpha} \in a^{-1}H\cap (aH)^{\alpha}$ for $\alpha\in K$ of types (2) and (3). The left cosets of $H$ are the blocks of $K$, so $aH =(aH)^{\alpha}$ for every $\alpha\in K$ of types (1) and (4),  $a^{-1}H=(aH)^{\alpha}$ for every $\alpha\in K$ of types (2) and (3). Applying Lemma \ref{2factor}, we conclude that   elements of types (1) and (4) fix all left cosets of $H$ while  elements of types  (2) and (3) interchange cosets  $a^iH$ and $a^{-i}H$.

Let $\alpha \in K$ be an automorphism of $\mathcal{A}$ that fixes $a^2 \in T_2$. The elements
$$a^3,~a,~ a^{3^{n-2}+3}b^2,~a^{2\cdot3^{n-2}+3}b,~ ab^2,~ ab$$
are adjacent to  $a^2$ in  $R(Z_4)$. The element $a$ is the only one of them from $Z_4$, the element $a^3$ is the only one of them from $X_1$. Therefore $a$ and $a^3$  are fixed by $\alpha$.

Further, we prove that $\alpha$ fixes $a^i$ for every $i=1,\ldots,3^{n-1}-1$. We proceed by induction on $i$. Suppose that $\alpha$ fixes $a^j, j\leq i,$. The elements
$$a^{i+1},~a^{i-1},~ a^{k_1}b^{l_1},~a^{k_2}b^{l_2},~a^{k_3}b^{l_3},~a^{k_4}b^{l_4},~ l_m\neq 0,~ m = 1,\ldots,4$$
are adjacent to $a^i$ in  $R(Z_4)$. They are permuted by $\alpha$ because  $\alpha$ fixes $a^i$.
The elements
$$a^{i+1},~a^{i+1+3^{n-2}},~ a^{i+1+2\cdot3^{n-2}},~a^{i-5+3^{n-2}},~a^{i-5+2\cdot3^{n-2}},~ a^{i-5}$$ are adjacent to $a^{i-2}$ in  $R(X_1)$. They are also permuted by $\alpha$ because $\alpha$ fixes $a^{i-2}$. However, $\alpha$ can not take $a^{i+1}$ into $a^{i-1}$ under action of $\alpha$ because $\alpha$ fixes $a^{i-1}$. Therefore

$${(a^{i+1})}^{\alpha} \in C \cap D,$$ where
$$C=\{a^{i+1},~ a^{k_1}b^{l_1},~a^{k_2}b^{l_2},~a^{k_3}b^{l_3},~a^{k_4}b^{l_4},~ l_m\neq 0,~ m = 1,\ldots,4\},$$
$$D=\{a^{i+1},~a^{i+1+3^{n-2}},~ a^{i+1+2\cdot3^{n-2}},~a^{i-5+3^{n-2}},~a^{i-5+2\cdot3^{n-2}},~ a^{i-5}\}.$$
Since $|C\cap D|=1$, we have ${(a^{i+1})}^{\alpha} = a^{i+1}$.

The element $\alpha$ can act non-trivially only by permuting  $a^ib$ and $a^ib^2$ for some $i$ because the left cosets of $B$ are the blocks of $K$. Therefore $\alpha$ fixes each left coset of $H$ as a set  (they are the blocks of $K$) and $\alpha$ is not of type (4) because it fixes the elements  $a^{3^{n-2}}$ and $a^{2\cdot3^{n-2}}$. Therefore $\alpha$ is of type (1). Thus $\alpha$ acts on $H$ trivially, in particular,  it fixes the elements $b$ and $b^2$. Note that $\alpha$  acts trivially on $Z_4$ and  $X_k$.

 The element $\alpha$ can act non-trivially on $Z_5$ only by permuting  $ab$ and $ab^2$ or $a^{-1}b$ and $a^{-1}b^2$. However, $ab$ and $a^{-1}b$ are adjacent to  $b$ in $R(Z_4)$ when the elements  $ab^2$ and $a^{-1}b^2$ are not adjacent to~$b$ in $R(Z_4)$. Therefore $\alpha$ acts trivially on $Z_5$. The elements of the form $a^{3l}b$ from $Y_k$ are adjacent to $b$ in  $R(X_k)$, the elements of the form $a^{3l}b^2$ from $Y_k$ are not adjacent to  $b$ in $R(X_k)$. Therefore $\alpha$ acts trivially on~$Y_k$.

Two elements from $T_i$, which can be permuted by $\alpha$, must be of the form $a^ib$ and $a^ib^2$, where $i$ is not a multiple of 3. If $i=3k+1$, then the element $a^ib$ is adjacent in $R(Z_4)$  to the element $a^{3k}b$, which is fixed by $\alpha$, the element $a^ib^2$ is not adjacent to the element $a^{3k}b$ in $R(Z_4)$. The case when $i=3k-1$ is similar. This means that $\alpha$ acts on $T_i$  trivially.

Thus, $\alpha = 1$ and the stabilizer of $a^2$ in $K$ is trivial. Then $\left|K\right| = \left|K_{a^2}\right|\cdot\left|a^2K\right| = 18$. We have a contradiction because the lengths of the orbits $Z_5$ and  $Y_k$ of~$K$ equal to 12 do not divide the order of $K$ equal to 18.

\end{proof}

The main theorem follows from Proposition  \ref{M27} and Proposition \ref{M3n}.

\end{document}